\documentclass{amsart}

\usepackage{subfigure}
\usepackage{amsthm}
\usepackage{amsmath}
\usepackage{url}
\usepackage{graphicx}
\usepackage{latexsym}
\usepackage{amsfonts}
\usepackage{amssymb}
\usepackage{verbatim}
\usepackage{hyperref}
\usepackage[font=small,format=plain,labelfont=bf,up,textfont=it,up]{caption}
\usepackage{enumitem}

\newtheorem{theorem}{Theorem}[section]

\newtheorem{definition}[theorem]{Definition}

\def\EL{\mathop{\rm EL}}

\def\ddt{\frac{d}{dt}}
\def\calt{\mathcal{T}}
\def\calu{\mathcal{U}}

\begin{document}



\title{Discrete Extremal Length and Cube Tilings in Finite Dimensions}

\author{William E. Wood \\University of Northern Iowa \\bill.wood@uni.edu}

\date{}

\keywords{discrete conformal geometry, tiling by cubes, extremal length, triple intersection property}

\subjclass[2000]{Primary: 53A30; Secondary: 57Q15, 57M15, 52C26}

\maketitle

\begin{abstract}
Extremal length is a conformal invariant that transfers naturally to the discrete setting, giving square tilings as a natural combinatorial analog of conformal mappings.  Recent work by S. Hersonsky has explored generalizing these ideas to three-dimensional cube tilings.  The connections between discrete extremal length and cube tilings survive the dimension jump, but a condition called the triple intersection property is needed to generalize existence arguments.  We show that this condition is too strong to realize a tiling, thus showing that discrete conformal mappings are far more limited in dimension three, mirroring the classical phenomenon.  We also generalize  results about discrete extremal length beyond dimension three and introduce some necessary conditions for cube tilings.

\end{abstract}

\section{Introduction}

Discrete conformal geometry was essentially born when William Thurston recognized that the existence of circle packings could be interpreted in terms of approximating a Riemann map \cite{andreev,koebe,thurston}.  That idea has since grown into a rich discrete theory mirroring classical conformal geometry.  For details, see e.g. \cite{ken} and the references therein.

 One powerful conformal invariant is extremal length \cite{ahlfors}. For a topological quadrilateral (a topological disk with four vertices on the boundary designated as the vertices of the quadrilateral), extremal length essentially measures the aspect ratio of the rectangle to which it can be mapped conformally and with vertices preserved.  The definition carries to the discrete setting in a natural way, and it too has a nice geometric interpretation.  Schramm and Cannon, Floyd, and Parry \cite{schramm,cfp} showed that the extremal metric is realized by a tiling by squares of a rectangle whose aspect ratio is the extremal length. The circle packing and square tiling models provide alternate paths to discretization of conformal geometry and behave very differently as convergence to the Riemann map and preservation of extremal length separate into geometrically distinct notions in the discrete setting (c.f. \cite{me,me2,dumbbells}).
 
 Extremal length generalizes naturally to higher dimensions. Our goal in this paper is  to pursue the connection between extremal length and square (now cube) tilings beyond two dimensions.  
 
 We should expect some trouble, as classical conformal geometry loses much of its depth in dimension three.  Nonetheless, Hersonsky was able to show in \cite{hers} that three-dimensional cube tilings do indeed realize the extremal length of a discrete parallelepiped, closely following Schramm's approach.   
 
 To prove that a square tiling exists realizing any discrete quadrilateral, Schramm used planarity in an essential way.  Moving to three dimensions, Hersonsky captured this use of planarity in a condition called the triple intersection property.  A cube tiling realizing the extremal length  of a discrete parallelepiped will exist if the triple intersection property is satisfied, whereas no conditions on the graph are necessary in dimension two.
 
 We will show that the triple intersection property is too strong to realize a cube tiling and thus that the two-dimensional arguments will not generalize.  This shows that discrete conformal mappings in dimension three are far more restrictive than in dimension two, mirroring the classical result that three-dimensional conformal mappings are M\"obius transformations.  Nonetheless, cube tilings certainly do exist and we explore some necessary conditions on a graph to realize a tiling.  Along the way we will also generalize some of Schramm's and Hersonsky's arguments to arbitrary finite dimensions.

\section{Cube Tilings and Extremal Metrics}\label{sec:background}

A \emph{discrete box} $\calt= \{G,{B_1,\ldots, B_n}, {B_1', \ldots B_n'}\}$ is a graph $G=(V,E)$ realizing the 1-skeleton of a triangulation of an $n$-dimensional cube whose boundary is the union of faces $\{B_1, B_1', \ldots, B_n,B_n'\}$, each of which is itself a triangulation of an $(n-1)$-dimensional cube and such that $B_i \cap B_i' = \varnothing$ for all integers $0\leq i\leq n$.  $B_1$ and $B_1'$ are distinguished as the \emph{top} and \emph{bottom} faces, respectively, and for general $i$ the faces $B_i$ and $B_i'$ are \emph{opposing} faces.

A \emph{geometric box} $\calu= \{{F_1,\ldots, F_n}, {F_1', \ldots F_n'}\}$ is a rectangular hyperparallelepiped positioned so that its face pairs $F_i$ and $F_i'$ are parallel to the coordinate axes and to each other, all coordinates of all points in $\calu$ are non-negative, and $\calu$ contains the origin as a vertex. The notation for a geometric box is chosen to be compatible with being triangulated by a discrete box with the notation above.

Let $\calt$ be a discrete box and let $\Gamma$ be the set of all vertex paths connecting $B_1$ to $B_1'$.  A \emph{metric} is a function $m: V\to [0,\infty)$.  The volume of a metric $m$ is $\|m\|_n = (\sum_{v\in V} m(v)^n)^{1/n}$. The \emph{length} of a vertex path $\gamma$ in $G$ with respect to a metric $m$ is $\ell_m(\gamma) = \sum_{v\in\gamma}m(v)$.  The length of a metric is $\ell_m=\inf_{\gamma\in\Gamma}\ell_m(\gamma)$ and its normalized length is $\hat{\ell}_m = \ell_m/\|m\|_n$. Define the \emph{extremal length} of $\calt$ to be the largest possible normalized length, i.e. \[\EL(\calt)=
\sup_{m\in\Lambda} \hat{\ell}_m
\]
where $\Lambda$ is the set of metrics with positive volume. 
This definition is analogous to the classical conformally invariant extremal length.

Note that the dimension $n$ and the choice of top and bottom faces are included in the definition of extremal length.  A metric is extremal if it realizes the extremal length.  The following theorem shows that these metrics always exist and are unique up to the scale invariance inherent in the definition.

\begin{theorem}
Every discrete box has a unique  extremal metric of unit volume.
\end{theorem}

\begin{proof}This is a special case of a theorem proved in \cite{cfp} that applies to any set of curves in any finite graph.  One forms a finite-dimensional vector space over the vertices of the graph and recognizes a metric as an element of this vector space.  By normalizing to unit area, searching for an extremal metric amounts to maximizing a continuous function on the unit sphere.  Existence follows from compactness of the sphere and uniqueness from convexity.
\end{proof}

A \emph{cube tiling} of a geometric box $\calu$ by a discrete box $\calt$ is a collection of $n$-dimensional cubes $C_v$, $v\in V$, such that $\cup_{v\in V} C_v = \calu$ and any two cubes that intersect do so along their boundaries.  If two vertices $v$ and $w$ are adjacent in $\calt$, then $C_v \cap C_w \neq \varnothing$.

We now define a convenient tool  connecting cube tilings to their corresponding discrete boxes.  For any point $ p \in F_1$, let $L_{p}$ be the line through $p$ and perpendicular to $F_1$.  Then define the discrete line $\gamma_{ p} = \{v\in V: C_v\cap L_{p} \neq \varnothing\}$. The discrete line defines a vertex path from $B_1$ to $B_1'$.

A cube tiling naturally defines a metric $s$ defined so that $s(v)$ is the side length of the cube $C_v$.  The main connection to extremal length is made by the following.

\begin{theorem}
The metric associated to a cube tiling of a geometric box is an extremal metric for the corresponding discrete box.\label{th:extremal}
\end{theorem}

\begin{proof} We generalize arguments from \cite{schramm} and \cite{hers}.   Let $\calu$ be the geometric box $[0,h_1]\times\cdots\times[0,h_n]$ with cube tiling $\{C_v\}$ and associated discrete box $\calt$ and metric $s$.  Let $m$ be extremal for $\calt$.  Assume we have rescaled so that $h_1 h_2 \cdots h_n = 1$ and thus $\|s\|_n = 1$.

By construction, $\ell_m \leq \ell_m(\gamma_p)$ for any $p\in F_1=[0,h_1]\times[0,h_2]\times\cdots\times[0,h_{n-1}]$.  Integrating both sides over $F_1$ with respect to Lebesgue measure $\mu$ gives $$\int_{F_1} \ell_m \ d\mu \leq \int_{F_1} \ell_m(\gamma_p) \ d\mu = \int_{F_1} \sum_{v\in\gamma_p} m(v) \ d\mu. $$
The left-hand side is constant.  To evaluate the integral on the right, note that $v \in \gamma_p$ if and only if $p$ lies in the $(n-1)$-dimensional shadow of $C_v$ in $F_1$.  In other words, $C_v$ contributes its measure in proportion to the volume $s(v)^{n-1}$ of one of its faces. Noting also that the volume of $F_1$ is $h_1 h_2\cdots h_{n-1}$, we have
$$\ell_m h_1 h_2\cdots h_{n-1} \leq \sum_{v\in V} m(v) s(v)^{n-1}.$$
Now employ H\"older's Inequality with $p=n$, $q=\frac{n}{n-1}$, giving
$$\ell_m h_1 h_2\cdots h_{n-1}  \leq(\sum_{v\in V} m(v)^n)^{1/n} (\sum_{v\in V} (s(v)^{n-1})^{\frac{n}{n-1}})^{\frac{n-1}{n}} = \| m \|_n \| s \|_n^{n-1} =\| m \|_n.$$

  We also have $\|s\|_n = h_1 h_2\cdots h_{n} $ and $\ell_s = h_n$, so
$$\hat\ell_m = \frac{\ell_m}{\| m \|_n} \leq \frac{ \| m \|_n}{h_1 h_2\cdots h_{n-1}\|m\|_n} = \frac{1}{h_1 h_2\cdots h_{n-1}}=\frac{\ell_s}{\|s\|_n} = \hat\ell_s.$$
Since $m$ is extremal, $\hat\ell_m$ is maximal and thus so too must be $\hat\ell_s$.  Therefore $s$ is an extremal metric for $\calu$. \end{proof}

\section{Existence of Cube Tilings: Sufficiency}\label{sec:sufficiency}

We would like a converse to Theorem~\ref{th:extremal} that constructs a cube tiling from an extremal metric.  Schramm proved in \cite{schramm} that in two dimensions any extremal metric will indeed yield a square tiling and moreover found an algorithm for its construction.  The proof depends on an inequality that is trivially satisfied in the planar setting but does not clearly generalize.  Hersonsky \cite{hers} developed a condition to serve as a proxy for planarity in a three-dimensional version of Schramm's proof.  We will show that this condition is too restrictive to admit a cube tiling in any non-planar case. Since our results will be negative, it is sufficient to restrict our attention to the case $n=3$ as the trouble will clearly extend to higher dimensions.

An important tool will be metrics that are extremal except along a path $\gamma$ which is given additional weight.  For an extremal metric $m$, path $\gamma$, and $t\geq 0$, define
\begin{displaymath}
   m_{\gamma,t}(v) = \left\{
     \begin{array}{lcr}
       m(v)+t & : & v \in \gamma\\
       m(v) & : &v \notin \gamma.
     \end{array}
   \right.
\end{displaymath} 
The following condition isolates the key step.
\begin{definition} A discrete  box satisfies \emph{Schramm's condition} if for all $1\leq i \leq n$
there exist shortest paths  $\alpha_i$ connecting $B_i$ to $B_i'$ satisfying the inequality $$\ddt(\ell_{m_{\alpha_i, t}})\big|_{t=0^+} \geq 1.$$
\end{definition}

This inequality can be used to show that shortest paths connecting opposing faces are long enough to reach the other side of the geometric box we are trying to tile.  Topological arguments and a volume calculation then show that the cubes must form a tiling.

\begin{theorem}
A discrete box satisfying Schramm's condition admits a tiling by cubes whose side lengths are determined by the extremal metric.
\end{theorem}

Schramm first presented the details of this argument in dimension two in \cite{schramm}.  Hersonsky was able to generalize the rest of Schramm's proof to dimension three in \cite{hers} (with the added assumption that the top and bottom faces be squares), and indeed the arguments could be generalized to finite dimensions with some extra bookkeeping.  The problem is thus reduced to satisfying Schramm's condition.

To meet Schramm's condition,  any shortest path connecting the top to the bottom must intersect each of the $\alpha_i$.  To see this, note that $\ell_m$ captures the shortest of all such paths and $\ell_{m_{\alpha, t}}$ will not depend on $t$ if any one top-to-bottom path misses $\alpha$, obliterating the derivative.

Schramm's condition is trivially met in dimension two, since any path connecting two opposing sides of a topological quadrilateral will necessarily intersect any path connecting the other two sides.  This is how Schramm used planarity in his proof.  In  generalizing this work to dimension three, Hersonsky created the following condition.

\begin{definition} A discrete three-dimensional box satisfies the \emph{triple intersection property} if for $i=2,3$ there exist shortest paths $\alpha_i$ connecting $B_i$ to $B_i'$ meeting all shortest paths connecting $B_1$ to $B_1'$.
\end{definition}

Unfortunately, our next result shows that this condition is too strong to realize a tiling.

\begin{theorem} The discrete three-dimensional box associated to a cube tiling cannot satisfy the triple intersection property.\label{th:tip}
\end{theorem}

\begin{proof} Let $\{C_v\}$ be a cube tiling of three-dimensional geometric box $\calu$ by discrete box $\calt$.  Suppose $\alpha$ and $\beta$ are shortest paths in $\calt$ connecting $B_2$ to $B_2'$ and  $B_3$ to $B_3'$, respectively, so that every  shortest path from  $B_1$ to $B_1'$ intersects both $\alpha$ and $\beta$.  

Consider the orthogonal projection of the cubes corresponding to vertices in $\alpha$ onto $F_1$.  This is a path of squares whose shared sides are all parallel because $\alpha$ is a shortest path.  The union of these squares must be all of $F_1$.  To see this, consider some point $p\in F_1$ not in one of these squares.  Then the discrete line $\gamma_p$ misses $\alpha$, contradicting the triple intersection property.

The only way a collection of squares whose intersecting edges are parallel can fill the rectangular $F_1$ is if they are congruent squares laid in a line, in this case connecting $F_2$ to $F_2'$.  But this violates the assumption that $B_3$ and $B_3'$ are disjoint, since this configuration connects $F_3$ to $F_3'$ with a single cube.
\end{proof}

Notice that our proof  never referenced the path $\beta$, meaning that requiring even  one cross-path to hit all top-bottom paths is prohibitive.  The argument clearly carries to higher dimensions as well.

The upshot of Theorem~\ref{th:tip} is that Schramm's condition, while trivially satisfied in dimension two, is too strong to induce analogous tilings in higher dimensions.  Since this condition captured a key piece of the argument, we see that higher dimensional discrete conformal geometry will have different behavior and requrire different tools  than in two dimensions.

\begin{figure}
\begin{center}
\includegraphics[width=3in]{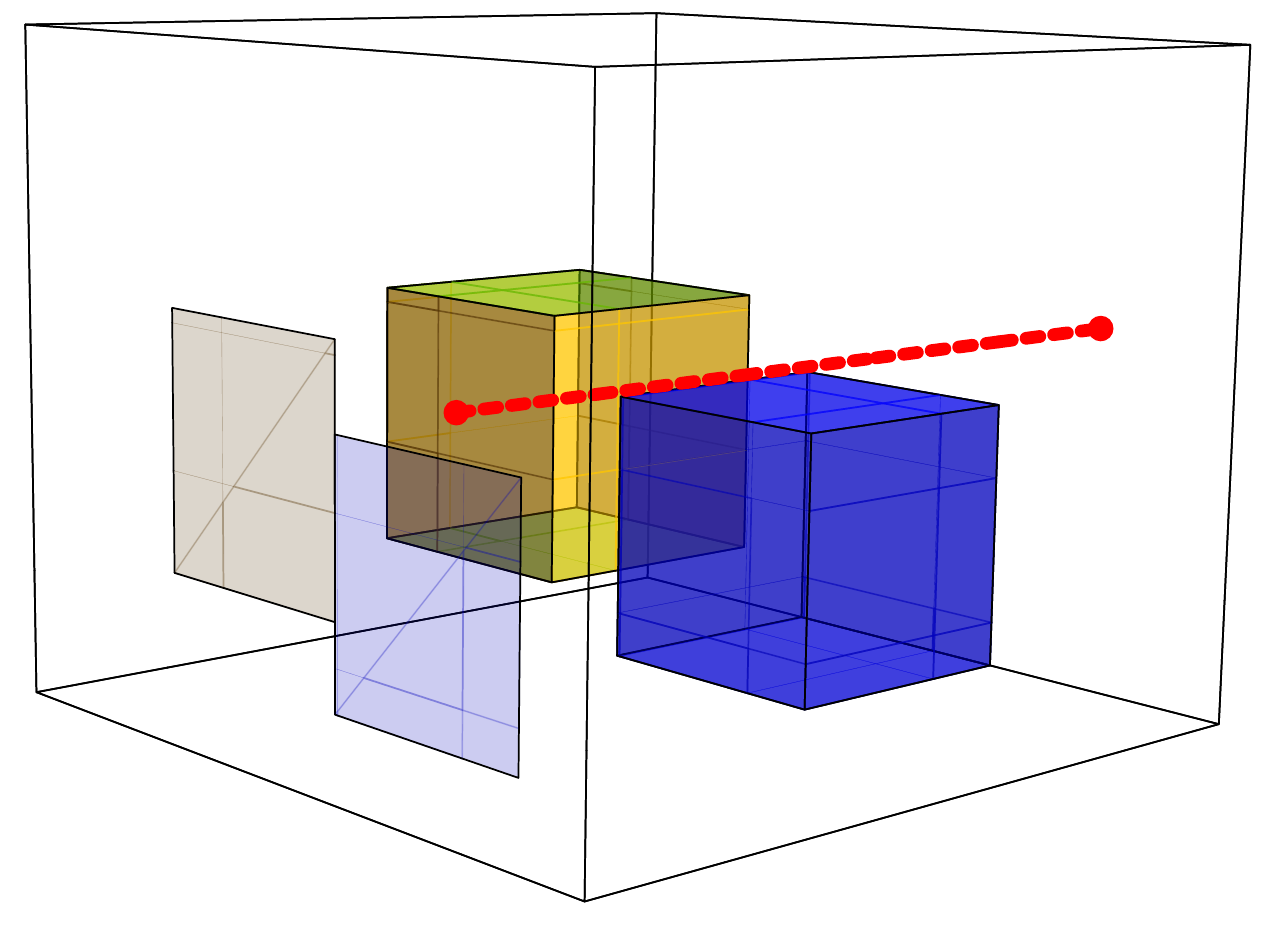}
\end{center}
\caption{The proof of Theorem~\ref{th:tip}.  The shadow of a shortest path on the bottom face (depicted here in the front) shows how to find a discrete line missing the path.}
\end{figure}

\section{Existence of Cube Tilings: Necessity}\label{sec:necessity}

We can  explicitly construct cube tilings.  For example, start with a geometric box and decimate it with equally spaced hyperplanes in all directions to form a grid of congruent cubes.  Then find collections of these cubes whose union is another smaller geometric box and glue the cuts back.  Repeating this process indefinitely on these smaller boxes can produce cube tilings with great combinatorial complexity. The contacts graph of such tilings (perhaps with some edges and null-weight vertices added to make it a triangulation)   will give extremal metrics by Theorem~\ref{th:extremal}, but there is no known method to recover the tiling from the combinatorics.  We close with a collection of necessary conditions for cube tiling.

\begin{theorem}\label{th:existence} Let $\calt = \{G, B_1,\ldots, B_n, B_1',\ldots, B_n\}$, $n\geq 3$,  be a geometric box with extremal metric $m$.  The following are necessary conditions for  $\calt$ to realize a cube tiling of a unit-volume geometric box $\calu = [0,h_1]\times\cdots\times[0,h_n]$.
\begin{enumerate}
\item \label{ec:cong} $\EL(B_i) = \EL(B_i')$ for all integers $1\leq i \leq n$, with $B_i, B'_i$ considered as discrete $(n-1)$-boxes with top/bottom faces chosen consistently.
\item \label{ec:prod}Let $\sigma_i$ be any permutation of $\{1,\ldots,n\}$ mapping $1$ to $i$ and define the discrete box $\calt_i= \{G,B_{\sigma_i(1)},\ldots, B_{\sigma_i(n)},B'_{\sigma_i(1)},\ldots, B'_{\sigma_i(n)}\}$.  That is, the $\calt_i$ are all possible choices of top/bottom faces for $\calt$.  Then $$\EL(\calt_1)\cdot\EL(\calt_2)\cdots\EL(\calt_n) = 1.$$
\item\label{ec:perfect} There exist distinct vertices $v$ and $w$ in $G$ with $m(v)=m(w)>0$.
\end{enumerate}
\end{theorem}
\begin{proof}

\begin{enumerate}
\item Opposing faces of a geometric box must be congruent, so this result follows by viewing the faces of $\calt$ as discrete boxes and the tiled faces of $\calu$ as $(n-1)$-dimensional cube tilings.
\item Theorem~\ref{th:extremal} shows that extremal length of a cube tiling is the distance between the top and bottom faces in the tiling, i.e. the side length of an edge connecting these faces.  The product of all such side lengths is the volume, which we assumed is one.
\item This repurposes the fact that there are no perfect cube tilings (tilings with no congruent cubes) to the setting of extremal length.  See \cite{tutte}.

\end{enumerate}
\end{proof}

Theorem~\ref{th:existence} suggests approaches to working with higher-dimensional extremal length by exploiting properties of geometric boxes and their tilings.  The first two parts are  compatibility conditions  relating extremal length of a discrete box   to the extremal lengths and volumes of its faces.  They are relevant to the three-dimensional case because two-dimensional square tilings always exist and are computable.  Part~\ref{ec:perfect} suggests how geometric and combinatorial  properties of cube tilings can be interpreted as statements about extremal length.



\begin{thebibliography}{15}


\bibitem{andreev} E. M. Andreev, \emph{Convex Polyhedra in Lobacevskii Space}, Mat. Sbornik, Volume 81, No. 123, 1970.


\bibitem{koebe} P. Koebe, \emph{Kontaktprobleme der Konformen Abbildung}, Ber. Sächs. Akad. Wiss. Leipzig, Math.-Phys. Kl. 88: 141–164, 1936.


\bibitem{thurston} W. Thurston, The finite Riemann mapping theorem, Invited talk at the International Symposium at Purdue University on the occasion of the proof of the Bieberbach conjecture, 1985.



\bibitem{ken} K. Stephenson, \emph{Introduction to Circle Packing}, Cambridge University Press, 2005.


\bibitem{ahlfors}L. Ahlfors, \emph{Conformal invariants: Topics in geometric function theory}, McGraw-Hill, Inc., 1973.


 

\bibitem{schramm} O. Schramm, \emph{Square tilings with prescribed combinatorics}, Israel Journal of Mathematics, Volume 84, Issue 1-2, pp. 97-118, 1993.

 \bibitem{cfp} J. W. Cannon, W. J. Floyd, and W. R. Parry,\emph{Squaring rectangles: the finite Riemann mapping theorem},	Contemp. Math., vol. 169, Amer. Math. Soc., Providence, RI,  pp. 133–-212, 1994.



\bibitem{me} W. Wood, \emph{Combinatorial Modulus and Type of Graphs},  Topology and its applications, 159, no. 17, pp. 2747--2761, 2009.

\bibitem{me2} W. Wood, \emph{Bounded outdegree and extremal length on discrete Riemann surfaces}, Conformal Geometry and Dynamics, 14, 194-201, 2010.  


\bibitem{dumbbells} J. Cannon, W. Floyd, W. Parry, \emph{Squaring rectangles for dumbbells},
 Conformal Geometry and Dynamics, 12, 109-132, 2008.


\bibitem{hers} S. Hersonsky, \emph{The triple intersection property, three dimensional extremal length, and tiling of a
topological cube}, Topology and its applications, 159, no. 10--11, pp. 2795-2805, 2012.


\bibitem{tutte}
R. Brooks, C. Smith, A. Stone,  and W. Tutte, \emph{The Dissection of Rectangles into Squares}, Duke Math. J. 7, pp. 312-–340, 1940.


\end{thebibliography}
\end{document}